\documentclass[a4paper,final]{amsart}

\usepackage{amsmath, amsthm, amssymb}
\usepackage[english]{babel}
\usepackage[utf8]{inputenc}
\usepackage{mathrsfs}                     
\usepackage{bbm}                          
\usepackage[all]{xy}                      
\usepackage{graphicx}                     
\usepackage{color}
\usepackage{url}
\usepackage{import}                       
\usepackage{todonotes}                    
\usepackage{algpseudocode}                
\usepackage[pdfdisplaydoctitle=true,
            colorlinks=true,
            urlcolor=blue,
            citecolor=blue,
            linkcolor=blue,
            pdfstartview=FitH,
            pdfpagemode= UseNone,
            bookmarksnumbered=true]{hyperref}
\usepackage[style=numeric, minnames=3,
            doi=false, url=false, isbn=false,
            firstinits=true,
            sortcites=true,
            backend=biber]{biblatex}      
\usepackage{enumerate}                    
\usepackage{framed}                       

\swapnumbers

\theoremstyle{plain}

\newtheorem{proposition}[subsection]{Proposition}
\newtheorem{lemma}[subsection]{Lemma}
\newtheorem{corollary}[subsection]{Corollary}
\newtheorem{conjecture}[subsection]{Conjecture}
\newtheorem*{theorem*}{Theorem}
\newtheorem*{proposition*}{Proposition}
\newtheorem*{lemma*}{Lemma}
\newtheorem*{corollary*}{Corollary}

\theoremstyle{definition}
\newtheorem{definition}[subsection]{Definition}

\newtheorem{example}[subsection]{Example}
\newtheorem*{definition*}{Definition}
\newtheorem*{remark*}{Remark}
\newtheorem*{example*}{Example}
\newtheorem*{openquestion*}{Open Question}

\def\al{\alpha}

\def\ga{\gamma}

\def\ep{\varepsilon}

\def\et{\eta}
\def\th{\theta}

\def\ta{\tau}
\def\ph{\varphi}
\def\ch{\chi}
\def\ps{\psi}

\def\De{\Delta}

    \let\ii=\i           
\def\inv{^{-1}}
\def\x{\times}
\def\p{\partial}

\def\R{{\mathbb R}}

\def\one{\mathbbm{1}}

\let\on=\operatorname

\let\wh=\widehat

\let\mb=\mathbb
\let\mc=\mathcal
\let\mf=\mathfrak

\newcommand{\ud}{\,\mathrm{d}}

\newcommand{\fourier}{{\mathscr{F}}}
\newcommand{\schwartz}{{\mathscr{S}}}

\newcommand{\executeiffilenewer}[3]{%
\ifnum\pdfstrcmp{\pdffilemoddate{#1}}%
{\pdffilemoddate{#2}}>0%
{\immediate\write18{#3}}\fi%
}

\hypersetup{
    pdftitle={Riemannian geometry induced by the diffeomorphism group}, 
    pdfauthor={Martins Bruveris}, 
    pdfsubject={MSC 2010: 58D05, 35Q35}, 
    pdfkeywords={Sobolev metrics, Diffeomorphism groups, Space of embeddings, Space of curves, Riemannian metric}, 
    pdflang=EN 
    }

\begin{document}


\title[Riemannian geometry induced by the diffeomorphism group]{Riemannian geometry on spaces of submanifolds induced by the diffeomorphism group}
\author{Martins Bruveris}
\address{Department of Mathematics, Brunel University London,
  Ux\-bridge, UB8 3PH, United Kingdom}
\email{martins.bruveris@brunel.ac.uk}


\date{\today}

\begin{abstract}
The space of embedded submanifolds plays an important role in applications such as computational anatomy and shape analysis. We can define two different classes on Riemannian metrics on this space: so-called outer metrics are metrics that measure shape changes using deformations of the ambient space and they find applications mostly in computational anatomy; the second class that are defined directly on the space of embeddings using intrinsic differential operations and they are used in shape analysis. In this paper we compare for the first time the topologies and the geodesic distance functions induced by these the two classes of metrics.
\end{abstract}

\maketitle

\setcounter{tocdepth}{1}

\section{Introduction}

The space of embedded submanifolds is interesting mathematically as well as relevant in applications. Mathematically, it is a truly nonlinear infinite-dimen\-sio\-nal manifold, and it finds applications among other places in computational anatomy, where embedded surfaces describe the shape of organs, and in computer vision, where closed curves represent outlines of objects and in shape analysis, where the aim is to extract the information contained in the shape of objects.

To fix notation, let $M$ and $N$ be smooth manifolds without boundary and $M$ compact. The space $\on{Emb}(M,N)$ consists of smooth embeddings of $M$ into $N$. We can represent the space of embedded submanifolds as the quotient $B_e(M,N) = \on{Emb}(M,N) / \on{Diff}(M)$ of embeddings by the diffeomorphism group $\on{Diff}(M)$. This space is a smooth Fr\'echet manifold~\cite[{}44.1]{Michor1997}. For analytical reasons and because it is the most important case for applications we will restrict ourselves to $N=\mathbb R^d$. In Section~\ref{sec:comparison} we will work with $M=S^1$, in which case $B_e(S^1,\R^d)$ is the space of unparametrized, embedded curves in Euclidean space.

Already~\cite{Michor2007} described several classes of Riemannian metrics that can be defined on the space of curves~\cite{Bauer2014}. We will consider two classes in particular and prove, to our knowledge, the first result relating them to each other. 

The first class consists of Riemannian metrics that are induced by the action of the diffeomorphism group of the ambient space. We start with a right-invariant Riemannian metric $G^{\mc D}$ on $\on{Diff}_c(\R^d)$, the group of compactly supported diffeomorphisms. The diffeomorphism group acts on the space $\on{Emb}(M,\R^d)$ from the left via
\[
\on{Diff}_c(\R^d) \x \on{Emb}(M,\R^d) \to \on{Emb}(\R^d)\,,\quad
(\ph,q)\mapsto \ph \circ q\,.
\]
The left action induces a metric on $\on{Emb}(M,\R^d)$, given by the formula
\[
G^{\mc O}_q(u,u) = \inf_{X \circ q = u} G^{\mc D}_{\on{Id}}(X,X)\,.
\]
The interpretation of this formula is that the cost of a deformation $u$ is given by the cost
of deforming the ambient space as measured by $G^{\mc D}$. In other words, we consider the most cost-effective deformation $X$, that induces the given deformation $u$ along the submanifold $q$. This motivates the name \emph{outer metrics}: the metrics are defined in terms of deformations of the ambient (or outer) space.

The class of outer metrics is widely used in the large deformation matching framework~\cite{Beg2005, Miller2015}, both for matching curves as well as surfaces \cite{Glaunes2008,Charon2013}. The great practical advantage of these metrics is that all computations can be done in the ambient space, which remains fixed. The motion of the submanifolds is then recovered using the action of the diffeomorphism group. Mathematically these metrics have received less attention. In \cite{Michor2007} the authors computed a formula for the induced metric and described the geodesic equation and in \cite{Micheli2013} they computed the curvature. The lack of attention can partly be explained by the fact that the metric $G^{\mc O}$ tends to be complicated even when the original metric $G^{\mc D}$ is simple. The metric $G^{\mc O}$ is given in terms of a pseudo-differantial operator and an explicit formula exists only for the inverse of this operator, even when $G^{\mc D}$ is given in terms of a differential operator. Some results for Riemannian metrics given by Fourier multipliers exist \cite{Bauer2015,Bauer2017_preprint}, but Riemannian metrics defined by pseudo-differential operators have been left mostly untouched. An exception are the papers \cite{Cismas2016a, Cismas2016b}, where the author encountered such metrics on $\on{Diff}(S^1)$ when studying right-invariant metrics on semi-direct products of diffeomorphism groups.

The second class of metrics are Sobolev metrics with constant coefficients. We will consider these metrics only on the space $\on{Imm}(S^1,\R^d)$ of immersed curves. They are metrics of the form
\[
G_c^{\mc I}(u,v) = \int_{S^1} a_0 \langle u, v \rangle + a_1 \langle D_s u, D_s v \rangle + \dots +
a_n \langle D_s^n u, D_s^n v \rangle \ud s\,,
\]
with constants $a_0, a_n > 0$ and $a_j \geq 0$. We call $n$ the order of the metric. In the above equation $D_s u = \frac{1}{|c'|} u'$ denotes differentiation with respect to arc length and $\ud s = |c'| \ud \th$ integration with respect to arc length. These metrics can be defined on the slightly larger space $\on{Imm}(S^1,\R^d)$; because arc length differentation is a local operation, self-intersections of the curve $c$ do not represent a problem. We call them \emph{inner metrics} to emphasize the fact that they are defined using intrinsic operations directly on the space of curves.

Sobolev metrics on curves have been independently introduced by \cite{Michor2006c,Charpiat2007,Mennucci2007} and they have been studied in \cite{Michor2006c,Michor2007,Mennucci2008}. More recently \cite{Bruveris2014,Bruveris2015} showed that Sobolev metrics of order 2 and higher are metrically and geodesically complete and any two curves can be joined by a minimizing geodesic. A particular first order Sobolev metric~\cite{Srivastava2011}---although not one with constant coefficients---has been used in a wide range of applications of shape analysis~\cite{Laga2014,Joshi2013,Kurtek2013,Laborde2013}. Inner metrics have been generalized to manifold-valued curves \cite{Celledoni2016, LeBrigant2016_preprint, Su2014} and to higher-dimensional immersed manifolds \cite{Bauer2011b,Bauer2012d}.

\subsection*{Contributions}
The goal of this paper is to study the topology induced by the geodesic distance functions of outer metrics and to relate it to the geodesic distance functions of inner metrics. Let $G^{\mc D}$ be a Sobolev metric of order $s$ with $s > d/2+1$ on $\on{Diff}_c(\R^d)$, i.e., the inner product $G^{\mc D}_{\on{Id}}(\cdot, \cdot)$ defines the $H^s$-topology on the space $\mf X_c(\R^d)$ of compactly supported vector fields. The induced metric $G^{\mc O}_q$ is defined by restricting vector fields to the embedded submanifold $q(M)$. Comparing this to the trace map in Sobolev spaces,
\[
\on{Tr}_q: H^s(\R^d,\R^d) \to H^{s'}(M,\R^d)\,,\quad X \mapsto \on{Tr}_qX = X \circ q\,,
\]
with $s' = s-(d-m)/2$, $m=\dim M$, we expect the induced metric $G^{\mc O}$ to be a Sobolev metric of order $s'$. We show in Section~\ref{sec:outer} that this is true pointwise---$G^{\mc O}_q(\cdot,\cdot)$ induces the $H^{s'}$-topology on $T_q\on{Emb}(M,\R^d)$---as well as for the geodesic distance---$\on{dist}^{\mc O}$ induces the $H^{s'}$-topology on $\on{Emb}(M,\R^d)$. We also prove the following result, relating the geodesic distance functions of inner and outer metrics.

\begin{theorem*}
Let $n \geq 2$ and $s \geq n + (d-1)/2$ be the orders of the metrics $G^{\mc I}$ and $G^{\mc D}$ respectively and denote by $G^{\mc O}$ the metric on $\on{Emb}(S^1,\R^d)$ induced by $G^{\mc D}$. Then, given $c_0 \in \on{Emb}(S^1,\R^d)$ and $R>0$, there exists $C = C(c_0, R)$, such that
\[
\on{dist}^{\mc I}(c_1, c_2) \leq C \on{dist}^{\mc O}(c_1, c_2)\,,
\]
holds for all $c_1, c_2 \in B^{\mc O}(c_0, R)$.
\end{theorem*}

In the theorem $\on{dist}^{\mc I}$ and $\on{dist}^{\mc O}$ denote the geodesic distance functions with res\-pect to $G^{\mc I}$ and $G^{\mc O}$ on $\on{Emb}(S^1,R^d)$ respectively and $B^{\mc O}(c_0,R)$ is the metric ball around $c_0$ of radius $R$ with respect to the metric $\on{dist}^{\mc O}$. 

Together with~\cite[Lemma~4.2(1)]{Bruveris2015} this shows that the identity maps
\[
(\on{Emb}(S^1,\R^d), \on{dist}^{\mc O}) \to (\on{Emb}(S^1,\R^d), \on{dist}^{\mc I})
\to (\on{Emb}(S^1,\R^d), \| \cdot \|_{H^{n}(S^1)})
\]
are Lipschitz continuous on every metric ball; note that this is a stronger property than local Lipschitz continuity. We also have local Lipschitz continuity for the identity map in the reverse direction
\[
(\on{Emb}(S^1,\R^d), \| \cdot \|_{H^{n}(S^1)}) \to (\on{Emb}(S^1,\R^d), \on{dist}^{\mc I})
\to (\on{Emb}(S^1,\R^d), \on{dist}^{\mc O})\,.
\]

\subsection*{Structure} The paper is structured as follows. Section~\ref{sec:sobolev} collects known results about continuity of various maps in Sobolev spaces. The only new result is Lemma~\ref{lem:sobolev_comp_trans_cont}, which shows continuity of the transpose of composition. In Section~\ref{sec:extension_smooth} we show that we can smoothly extend functions defined submanifolds to the whole space and this extension map can be chosen to depend smoothly on the submanifold. In Section~\ref{sec:trace_sobolev} we consider the trace and extension operators in Sobolev spaces and show that they depend continuously on the submanifold in question. Section~\ref{sec:outer} is devoted to the study of outer metrics and contains most of the main results. Section~\ref{sec:comparison} summarises some known results on inner metrics on curves and uses them to prove the theorem about the comparison of inner and outer metrics.

\section{Sobolev space estimates}
\label{sec:sobolev}

The Sobolev spaces $H^s(\R^d)$ with $s \in \R$ can be defined in terms of the Fourier transform
\[ 
\fourier f(\xi) = (2\pi)^{-d/2} \int_{\R^d} e^{-i \langle  x,\xi\rangle} f(x) \ud x\,,
\]
and consist of temperate distributions $f$ with the property that $(1+|\xi|^2)^{s/2} \mc F f$ is $L^2$-integrable. An inner product on $H^s(\R^d)$ is given by
\[
\langle f, g \rangle_{H^s} = \mf{Re} \int_{\R^d} (1 + |\xi|^2)^s \fourier f(\xi) \overline{\fourier g(\xi)} \ud \xi\,.
\]

A large part of the analysis presented in this paper relies on estimates in Sobolev spaces. When defining smoothing operators in Lemma~\ref{ss:smoothing_operator} we will need an estimate on multiplication in Sobolev space; a proof can be found in~\cite[Proposition 25.1]{Treves1975}. Here $\schwartz(\R^d,\R)$ denotes the Schwartz space of rapidly decaying functions.

\begin{lemma}
\label{lem:sob_mult_estimate}
Let $s \in \R$. Then pointwise multiplication
\[
\schwartz(\R^d,\R) \x H^s(\R^d,\R) \to H^s(\R^d,\R)\,,\quad (\ph, f) \mapsto \ph \cdot f\,,
\]
is a continuous, bilinear map and we have the estimate
\[
\| \ph f \|_{H^s} \leq 2^{|s|} \| f \|_{H^s} \int_{\R^d} \langle \xi \rangle^{|s|} | \wh \ph(\xi)| \ud \xi\,.
\]
\end{lemma}

We will also need to multiply to Sobolev functions with each other. A proof can be found, for example, in \cite[Lemma~2.3]{Inci2013}. 

\begin{lemma}
Let $s, s' \in \R$ with $s > d/2$ and $0 \leq s' \leq s$. Then pointwise multiplication
\[
H^s(\R^d, \R) \times H^{s'}(\R^d,\R) \to H^{s'}(\R^d,\R)\,,\quad
(f, g) \mapsto f \cdot g\,,
\]
is a bounded bilinear map.
\end{lemma}


\subsection{The group $\mc D^s(\R^d)$ of Sobolev diffeomorphisms}

Denote by $\on{Diff}^1(\R^d)$ the group of $C^1$-diffeomorphisms of $\R^d$, i.e.,
\[
\on{Diff}^1(\R^d) = \{ \ph \in C^1(\R^d,\R^d) \,:\, \ph \text{ bijective, } \ph\inv \in C^1(\R^d,\R^d) \}\,.
\] 
For $s > d/2+1$ and $s \in \R$ there are three equivalent ways of defining the group $\mc D^s(\R^s)$ of Sobolev diffeomorphisms:
\begin{align*}
\mc D^s(\R^d) &= \{ \ph \in \on{Id} + H^s(\R^d,\R^d) \,:\, \ph \text{ bijective, }
\ph\inv \in \on{Id} + H^s(\R^d,\R^d) \} \\
&= \{ \ph \in \on{Id} + H^s(\R^d,\R^d) \,:\, 
\ph \in \on{Diff}^1(\R^d) \} \\
&= \{ \ph \in \on{Id} + H^s(\R^d,\R^d) \,:\, 
\det D\ph(x) > 0,\, \forall x \in \R^d \}\,.
\end{align*}
If we denote the three sets on the right by $A_1$, $A_2$ and $A_3$, then it is not difficult to see the inclusions $A_1 \subseteq A_2 \subseteq A_3$. The equivalence $A_1 = A_2$ has first been shown in \cite[Sect. 3]{Ebin1970b} for the diffeomorphism group of a compact manifold; a proof for $\mc D^q(\R^d)$ can be found in \cite{Inci2013}. Regarding the inclusion $A_3 \subseteq A_2$, it is shown in \cite[Cor. 4.3]{Palais1959} that if $\ph \in C^1$ with $\det D\ph(x) > 0$ and $\lim_{|x |\to \infty} | \ph(x)| = \infty$, then $\ph$ is a $C^1$-diffeomorphism.

It follows from the Sobolev embedding theorem, that $\mc D^s(\R^d) - \on{Id}$ is an open subset of $H^s(\R^d,\R^d)$ and thus a Hilbert manifold. Since each $\ph \in \mc D^s(\R^d)$ has to decay to the identity as $|x|\to \infty$, it follows that $\ph$ is orientation preserving. More importantly, $\mc D^s(\R^n)$ is a topological group, but not a Lie group, since left-multiplication and inversion are continuous, but not smooth operations \cite{Inci2013}.


We will make repeated use of the continuity of composition in Sobolev spaces. A proof of this result can be found in \cite[Lemma~2.7]{Inci2013}.

\begin{lemma}
\label{lem:sobolev_comp_cont}
Let $s > d/2+1$ and $0\leq s' \leq s$. Then composition
\[
R : \mc D^s(\R^d) \x H^{s'}(\R^d) \to H^{s'}(\R^d)\,,\quad
(\ph,f) \mapsto R_\ph f = f \circ \ph
\]
is continuous. Moreover, given  $M, C>0$, there exists a constant $C_{s'}=C_{s'}(M,C)$, such that if $\ph \in \mc D^s(\R^d)$ satisfies
\[
\inf_{x \in \R^d} \det D\ph(x) > M
\quad\text{and}\quad
\| \ph - \on{Id} \|_{H^s} < C\,,
\]
then for any $f \in H^{s'}(\R^d)$ one has
\[
\| f \circ \ph \|_{H^{s'}} \leq C_{s'} \| f \|_{H^{s'}}\,.
\]
\end{lemma}

We will also need the continuity of the transpose map. The proof is not difficult, but to our knowledge cannot be found in the literature.

\begin{lemma}
\label{lem:sobolev_comp_trans_cont}
Let $s > d/2+1$ and $0 \leq s' \leq s$. Then the map
\[
R^\ast : \mc D^s(\R^d) \x H^{-s'}(\R^d) \to H^{-s'}(\R^d)\,,\quad
(\ph, \al) \mapsto R_\ph^\ast \al\,,
\]
defined via $\langle R^\ast_\ph \al, f \rangle_{H^{-s'}} = \langle \al, R_\ph f \rangle_{H^{-s'}}$ 
is continuous.
\end{lemma}

\begin{proof}
Consider a sequence $(\ph_n, \al_n) \to (\ph, \al)$ in $\mc D^s(\R^d) \x H^{-s'}(\R^d)$ and let $f \in L^2(\R^d)$. Write
\[
R^\ast_{\ph_n} \al_n - R^\ast_{\ph} \al = R^\ast_{\ph_n}(\al_n - \al) + (R^{\ast}_{\ph_n} - R^\ast_\ph)(\al - f) + (R^\ast_{\ph_n} - R^\ast_\ph) f\,.
\]
Because $R : \mc D^s \x H^{s'} \to H^{s'}$ is continuous, its operator norm $\| R_\ph\|_{L(H^{s'},H^{s'})}$ is locally bounded. Let $C > 0$ be such that
\[
\| R_{\ph_n} \|_{L(H^{s'},H^{s'})} \leq C\,,\quad \forall n \in \mb N\,,
\]
and note that
$\| R_\ph\|_{L(H^{s'},H^{s'})} = \| R^\ast_\ph\|_{L(H^{-s'},H^{-s'})}$. Thus
\begin{align*}
\| R^\ast_{\ph_n} \al_n - R^\ast_{\ph} \al \|_{H^{-s'}}
&\leq  C \| \al_n - \al \|_{H^{-s'}} + 2C \| \al - f \|_{H^{-s'}} + \| (R^\ast_{\ph_n} - R^\ast_\ph) f\|_{L^2}\,.
\end{align*}
Given $\ep > 0$ we choose $f \in L^2$ sufficiently $H^{-s'}$-close to $\al$ and note that on $L^2$ we have
\begin{align*}
\langle R_\ph^\ast f, g \rangle_{L^2} &= 
\langle f, R_\ph g \rangle_{L^2} = \int_{\R^d} f(x) g(\ph(x)) \ud x \\
&= \int_{\R^d} f(\ph\inv(x)) g(x) \det D\ph\inv(x) \ud x \\
&= \langle (\det D\ph\inv).R_{\ph\inv} f , g \rangle_{L^2}\,.
\end{align*}
Hence $\ph_n \to \ph$ in $\mc D^s(\R^d)$ implies $R^\ast_{\ph_n}f \to R^\ast_\ph f$ in $L^2$ and thus also $R^\ast_{\ph_n} \al_n \to R^\ast_\ph \al$ in $H^{-s'}$.
\end{proof}

\subsection{Flows of diffeomorphisms}
A natural way to generate diffeomorphisms is via time-dependant vector fields. Let $I=[0,1]$ and $s > d/2+1$. We will consider vector fields $u \in L^1(I, H^s(\R^d,\R^d))$, i.e. integrable in time and $H^s$-regular in space. The flow of $u$, written $\ph = \on{Fl}(u)$ is a continuous curve $\ph \in C(I,\mc D^s(\R^d))$, satisfying
\[
\ph(t) = \on{Id}_{\R^d} + \int_0^t u(\ta) \circ \ph(\ta) \ud \ta\,,
\]
with the integral being the Bochner integral in $H^s$. We will also write
\[
\p_t \ph(t) = u(t) \circ \ph(t)\,;
\]
however this only holds $t$-a.e.. We will write $\ph(1) = \on{Fl}_1(u)$ for the flow at time $t=1$. Because composition in Sobolev spaces is not a Lip\-schitz continuous map, the existence of a flow is a nontrivial result.

\begin{lemma}\cite[Thm.~4.4]{Bruveris2014_preprint}
\label{lem:ex_flow}
Let $s > d/2+1$ and $u \in L^1(I,H^s(\R^d,\R^d))$. Then $u$ has a $\mc D^s(\R^d)$-valued flow and the map
\[
\on{Fl}: L^1(I,H^s(\R^d,\R^d)) \to C(I,\mc D^s(\R^d))\,\quad u \mapsto \ph\,,
\]
satisfying $\p_t \ph(t) = u(t) \circ \ph(t)$ and $\ph(0) = \on{Id}_{\R^d}$ is continuous.
\end{lemma}

We can estimate the $H^s$-norm of diffeomorphisms that are generated by flows. The following lemma was stated informally in \cite[Rem.~3.6]{Bruveris2014_preprint}.

\begin{lemma}
\label{lem:diff_ball_bound}
Let $s > d/2+1$. Given $r > 0$, there exist constants $M$ and $C$, such that the bounds
\[
\inf_{x \in \R^d} \det D\ph(t,x) > M\;\;\text{and}\;\;
\| \ph(t) - \on{Id} \|_{H^s} < C
\]
hold for diffeomorphisms $\ph \in \mc D^s(\R^d)$, that can be written as $\ph = \on{Fl}_1(u)$ with $\| u \|_{L^1} < r$.
\end{lemma}

Lemma~\ref{lem:diff_ball_bound} together with Lemma~\ref{lem:sobolev_comp_cont} imply the following.

\begin{lemma}
\label{lem:comp_diff_ball}
Let $s > d/2 + 1$ and $0 \leq s' \leq s$. Given $r > 0$, there exists a constant $C_{s'} = C_{s'}(r)$, such that the inequality
\[
\| f \circ \ph \|_{H^{s'}} \leq C_{s'} \| f \|_{H^{s'}}\,,
\]
holds for all $\ph \in \mc D^{s}(\R^d)$ that can be written as $\ph = \on{Fl}_1(u)$ with $\| u \|_{L^1} < r$ and all $f \in H^{s'}(\R^d)$.
\end{lemma}

Note that if $\ph = \on{Fl}_1(u)$, then $\ph\inv = \on{Fl}_1(v)$, where $v(t) = -u(1-t)$ and $\|v\|_{L^1(I,H^s)} = \|u \|_{L^1(I,H^s)}$. Hence under the same assumptions as in Lemma~\ref{lem:comp_diff_ball} we also have the inequality
\[
\| f \circ \ph\inv \|_{H^{s'}} \leq C_{s'} \| f \|_{H^{s'}}\,.
\]

\section{Smooth extension of maps}
\label{sec:extension_smooth}

Consider an embedding $q : M \to N$. It is well-known, that any function $f : M \to \R$ can be extended to a function $\tilde f : N \to \R$, such that $\tilde f \circ q = f$. This extension is of course not unique, but we can choose the extension map to depend smootly on the emdedding $q$.

\begin{proposition}
\label{prop:ext_function}
Let $M$ be a compact and $N$ a finite-dimensional manifold. Given $q \in \on{Emb}(M,N)$ there exists an open neighborhood $\mc U \subseteq \on{Emb}(M,N)$ of $q$ and a smooth map
\[
F : \mc U \x C^\infty(M,\R) \to C_c^\infty(N,\R)\,,
\]
linear in the second component, that satisfies
\[
F(r, f) \circ r = f\,,
\]
for all $r \in \mc U$ and $f \in C^\infty(M,\R)$.
\end{proposition}

\begin{proof}
We follow the proof of \cite[Thm.~44.1]{Kriegl1997} regarding the construction of an open neighborhood of $q$. Fix a Riemannian metric on $N$ and let $\on{exp}$ be its exponential map. Next, let $\pi : \mc N(q) \to M$ be the normal bundle of $q$, defined as $\mc N(q)_x = \left(T_x q(T_x M)\right)^\perp \subseteq T_{q(x)}N$ for $x \in M$; the orthogonal complement is taken with respect to the fixed Rie\-man\-nian metric on $N$. Then $\bar q$ is an injective vector bundle homomorphism over $q$:
\[ 
\xymatrix{
\mc N(q) \ar[d]_\pi \ar[r]^{\bar q} & TN \ar[d]^{\pi_{TN}} \\
M \ar[r]^q & N
}
\]

Let $U$ be a bounded open neighborhood of the zero section of $\mc N(q)$, small enough that $\on{exp} \circ \bar q : U \to N$ is a diffeomorphism onto its image; set $\ta = \on{exp} \circ \bar q$ and $V = \ta(U)$. We shrink the set $U$ to $U_2 = \frac 12 U$; boundedness of $U$ is necessary to have $\overline{U_2} \subsetneq U$. Set $V_2 = \ta(U_2)$. Define the open set
\[
\tilde{\mc U} = \{ r \in \on{Emb}(M,N) \,:\, r(M) \subseteq V_2 \}\,.
\]

Note that $\pi \circ \ta\inv \circ q(x) = x$ and we can consider the map $r \mapsto \ph(r) = \pi \circ \ta\inv \circ r \in C^\infty(M,M)$. Since $\on{Diff}(M)$ is open in $C^\infty(M,M)$ we have $\pi \circ \ta\inv \circ r \in \on{Diff}(M)$ for $r$ sufficiently close to $q$. The required open neighborhood is $\mc U = \tilde{\mc U} \cap \ph\inv\left(\on{Diff}(M)\right)$.

Let $\et \in C^\infty(N,\R)$ be a function satisfying $\et|_{V_2} \equiv 1$ and $\on{supp}(\et) \subseteq V$. Given $f \in C^\infty(M,\R)$, define $\tilde f \in C^\infty(V,\R)$ via $\tilde f = f \circ \ph(r)\inv \circ \pi \circ \ta\inv$. Then the map $f \mapsto \tilde f$ is smooth and we can define the extension operator as
\[
F(r,f)(y) = \et(y) \tilde f(y)\,.
\]
Note that this is well-defined and smooth. Furthermore we have for $r \in \mc U$,
\[
F(r,f) \circ r = f \circ \ph(r)\inv \circ \pi \circ \ta\inv \circ r 
= f \circ \ph(r)\inv \circ \ph(r) = f\,.
\]
This concludes the proof.
\end{proof}

\begin{corollary}
\label{cor:extend_emb_to_diff}
Let $M$ be a compact and $N$ a finite-dimensional manifold. Given $q \in \on{Emb}(M,N)$ there exists an open neighborhood $\mc U \subseteq \on{Emb}(M,N)$ of $q$ and a smooth map
\[
E : \mc U \to \on{Diff}_c(N)\,,
\]
such that $E(q) = \on{Id}_{\R^d}$ and
\[
r = E(r) \circ q\,.
\]
holds for all $r \in \mc U$.
\end{corollary}

\begin{proof}
We define
\[
E(r) = \on{Id}_{\R^d} + F(q, q-q)\,,
\]
with $F(q,f)$ as in Proposition~\ref{prop:ext_function}. Then $E(r) \circ q = r$ as required at it remains to verify that $E(r) \in \on{Diff}_c(\R^d)$. Because $\on{Id}_{\R^d} - \on{Diff}_c(\R^d)$ is open in $C^\infty_c(\R^d,\R^d)$, there exists a neighborhood $\mc V$ of the $0$-function, such that $r-q \in \mc V$ implies $\on{Id}_{\R^d} + F(q,r-q) \in \on{Diff}_c(\R^d)$. Now set $\mc U = q + \mc V$.
\end{proof}

\section{Trace theorem in Sobolev spaces}
\label{sec:trace_sobolev}

The general trace theorem for Sobolev spaces states that
\[
\on{Tr}_M H^s_p(N) = B^{s-\frac{d-m}p}_{p,p}(M)\,,
\]
where $M \subseteq N$ are manifolds of bounded geometry; a proof can be found in \cite{Grosse2013}. Here we are interested in the continuous dependance of the trace map on the submanifold $M$. We restrict ourselves to the case $M$ compact and $N=\R^d$.

\begin{lemma}
\label{lem:trace_theorem}
Let $M$ be a compact manifold with $\on{dim} M = m$ and $q \in \on{Emb}(M,\R^d)$. If $s > d/2$ and $s' = s-(d-m)/2$, then the trace map
\[
\on{Tr}_q : H^s(\R^d) \to H^{s'}(M)\,,\quad f \mapsto f \circ q\,,
\]
is a continuous operator. Furthermore, there exists a continuous extension operator
\[
\on{Ex}_q : C^\infty(M) \to C^\infty_c(\R^d)\,,
\]
such that for each $s$ and $s'$ as above it extends continuously to
\[
\on{Ex}_q : H^{s'}(M) \to H^s(\R^d)\,,
\]
and satisfies $\on{Tr}_q \circ \on{Ex}_q = \on{Id}_{H^{s'}(M)}$, i.e., $\on{Ex}_q(f) \circ q = f$.
\end{lemma}

\begin{proof}
The embedding $q$ remains fixed and $q : M \to q(M)$ is a diffeomorphism between smooth manifolds; thus we can regard $M$ as a submanifold of $\R^d$ with $q$ the canonical embedding. In this case the composition $\mu^q$ coincides with the trace map, $\mu^q(f) = \on{Tr}_M f = f|_M$. The boundedness of the trace map and the construction of an extension operator are shown in \cite[Theorem~4.10]{Grosse2013}. An examination of the proof shows that the extension operator constructed there maps smooth functions on $M$ to compactly supported functions on $\R^d$ and that the construction does not depend on the choice of $s$.
\end{proof}

We know show that the trace map depends continuously on the submanifold and the extension can be chosen to be continuous as well.

\begin{proposition}
\label{prop:trace_theorem_cont}
Let $M$ be a compact manifold with $\dim M = m$, $s > d/2$ and $s'=s-(d-m)/2$. Then the trace operator and its dual
\begin{align*}
\on{Tr} &: \on{Emb}(M,\R^d) \x H^s(\R^d) \to H^{s'}(M)\,,\quad
(q,f) \mapsto \on{Tr}_q f \\
\on{Tr}^\ast &: \on{Emb}(M,\R^d) \x H^{-s'}(M) \to H^{-s}(\R^d) \,,\quad
(q,\al) \mapsto \on{Tr}^\ast_q \al\,,
\end{align*}
are continuous maps and around each $q \in \on{Emb}(M,\R^d)$ there exists an open neighborhood $\mc U \subseteq \on{Emb}(M,\R^d)$ and an extension map, such that it and its dual 
\begin{align*}
\on{Ex} &: \mc U \x H^{s'}(M) \to H^{s}(\R^d) \,,\quad
(r,f) \mapsto \on{Ex}_r f\,, \\
\on{Ex}^\ast &: \mc U \x H^{-s}(\R^d) \to H^{-s'}(M) \,,\quad
(r,\al) \mapsto \on{Ex}^\ast_r \al \,.
\end{align*}
are continuous.
\end{proposition}

\begin{proof}
The inclusion $\on{Diff}_c(\R^d) \subseteq \mc D^s(\R^d)$ is smooth and therefore the map constucted in Corollary~\ref{cor:extend_emb_to_diff} is a smooth map $E: \mc U \to \mc D^s(\R^d)$. Since $E$ is a map between Fr\'echet spaces smoothness implies continuity.

Fix $q \in \on{Emb}(M,\R^d)$ and let $r \in \mc U$ with $\mc U$ given by Corollary~\ref{cor:extend_emb_to_diff}. Write
\begin{align*}
\on{Tr}_r f &= f \circ E(r) \circ E(r)\inv r = \on{Tr}_{q} R_{E(r)} f \\
\on{Tr}_r^\ast \al &= R_{E(r)}^\ast \on{Tr}_q^\ast \al\,,
\end{align*}
which shows using Lemma~\ref{lem:sobolev_comp_cont} and Lemma~\ref{lem:sobolev_comp_trans_cont} that $\on{Tr}$ and $\on{Tr}^\ast$ are continuous.

To construct the extension operators we proceed similarly. With $q$ fixed we let $\on{Ex}_q$ be the extension operator from Lemma~\ref{lem:trace_theorem} and we define for $r \in \mc U$,
\begin{align*}
\on{Ex}_r f &= R_{E(r)\inv} \on{Ex}_q f\,.
\end{align*}
Then $\on{Ex}$ and $\on{Ex}^\ast$ are continuous in $(r,f)$ and $(r,\al)$ respectively and we have
\[
\on{Tr}_r \on{Ex}_r f = \on{Tr}_q R_{E(r)} R_{E(r)\inv} \on{Ex}_q = \on{Id}_{H^{s'}(M)}\,.\qedhere
\]
\end{proof}

\section{%
\texorpdfstring%
{Outer metrics on $\on{Emb}(M,\R^d)$}%
{Outer metrics on Emb(M,Rd)}%
}
\label{sec:outer}

\begin{definition}
A \emph{right-invariant Sobolev metric of order $s$} is a Riemannian metric $G^{\mc D}$ on $\on{Diff}_c(\R^d)$ satisfying
\[
G^{\mc D}_{\ph}(X_\ph,Y_\ph) = G^{\mc D}_{\on{Id}}(X_\ph\circ\ph\inv, Y_\ph\circ\ph\inv)\,,
\]
for all $X_\ph,Y_\ph\in T_\ph\on{Diff}_c(\R^d)$ with the property that the inner product $G^{\mc D}_{\on{Id}}(\cdot,\cdot)$ induces on $\mf X_c(\R^d)$ the Sobolev $H^s$-topology.
\end{definition}

In fact, any Sobolev inner product $\langle \cdot, \cdot \rangle_{H^s}$ of order $s$ gives rise to a right-invariant Riemannian metric $G^{\mc D}$ via the formula
\[
G^{\mc D}_{\ph}(X_\ph,Y_\ph) = \langle X_\ph\circ\ph\inv, Y_\ph\circ\ph\inv \rangle_{H^s}\,.
\]
Because $\on{Diff}_c(\R^d)$ is a Lie group, this is defines a smooth Riemannian metric on $\on{Diff}_c(\R^d)$. Smoothness is a nontrivial question, if we wanted to extend $G^{\mc D}$ to the Sobolev completion $\mc D^s(\R^d)$, because the latter space is a topological group but not a Lie group. In the smooth category, however, we encounter no problems. There is a one-to-one correspondence between Sobolev inner products and right-invariant Sobolev metrics.

\subsection{Metrics on $\on{Diff}_c(\R^d)$}
Let $G^{\mc D}$ be a right-invariant Sobolev metric of order $s$ with $s \geq 0$ on $\on{Diff}_c(\R^d)$. Denote by
\[
A : H^s(\R^d,\R^d) \to H^{-s}(\R^d,\R^d)\,,
\]
the Riesz isomorphism associated to $G^{\mc D}_{\on{Id}}(\cdot, \cdot)$, i.e.,
\[
G^{\mc D}_{\on{Id}}(X, Y) = \langle AX, Y \rangle_{H^{-s}}\,,
\]
for all $X,Y \in \mf X_c(\R^d)$. We extend $G^{\mc D}_{\on{Id}}(\cdot,\cdot)$ to the completion $H^s(\R^d,\R^d)$ and we will use the notation
\[
\| X \|^2_A = G^{\mc D}_{\on{Id}}(X,X)\,.
\]
Because $\| \cdot\|_A$ induces the $H^s$-topology, there exists a constant $C > 0$, such that
\[
C\inv \| X\|_{H^s} \leq \| X \|_A \leq C \| X \|_{H^s}\,,
\]
holds for all $X \in H^s(\R^d,\R^d)$.

We denote by $\on{dist}^{\mc D}$ the geodesic distance induced by $G^{\mc D}$ on $\on{Diff}_c(\R^d)$. It is given by
\[
\on{dist}^{\mc D}(\ph_0, \ph_1) = \inf \left\{ L(\ph) \,:\, \ph \in C^\infty([0,1], \on{Diff}_c(\R^d)),\, \ph(0)= \ph_0,\, \ph(1)=\ph_1 \right\}\,,
\]
where
\[
L(\ph) = \int_0^1 \sqrt{G^{\mc D}_{\ph(t)}(\p_t \ph(t), \p_t \ph(t))} \ud t\,,
\]
is the length of a path.

\begin{lemma}
Let $s > d/2+1$. The geodesic distance $\on{dist}^{\mc D}$ of a right-invariant $H^s$-metric $G^{\mc D}$ on $\on{Diff}_c(\R^d)$ induces the $H^s$-topology on $\on{Diff}_c(\R^d)$, i.e.,
\[
\on{dist}^{\mc D}(\ph_n, \ph) \to 0
\quad\Leftrightarrow\quad
\| \ph_n - \ph\|_{H^s} \to 0\,.
\]
\end{lemma}

\begin{proof}
Assume $\| \ph_n - \ph \|_{H^s} \to 0$. Consider the linear path $\ps_n(t) = (1-t)\ph_n + t\ph$. For $n$ large enough, say $n \geq N$, and $t \in [0,1]$ we have $\ps_n(t) \in \on{Diff}_c(\R^d)$ and the set $\{ \ps_n(t)\inv \,:\, n \geq N,\, t \in [0,1]\}$ satisfies the assumptions of Lemma~\ref{lem:sobolev_comp_cont}. Hence
\[
L(\ps_n) = \int_0^1 \left\| (\ph-\ph_n) \circ \ps_n(t)\inv \right\|_{H^s} \ud t
\leq C \| \ph-\ph_n \|_{H^s}\,,
\]
for some constant $C$, thus showing $\on{dist}^{\mc D}(\ph_n, \ph) \leq C \| \ph_n - \ph\|_{H^s}$, which implies $\on{dist}^{\mc D}(\ph_n, \ph) \to 0$.

Now assume $\on{dist}^{\mc D}(\ph_n, \ph) \to 0$ and for each $n$, choose a smooth path $\ps_n \in C^\infty([0,1], \on{Diff}_c(\R^d))$ from $\ph$ to $\ph_n$ with $L(\ps_n) \leq \on{dist}^{\mc D}(\ph_n, \ph) + \frac 1n$. Define $u_n = \p_t \ps_n \circ \ps_n\inv$. Then $\| u_n \|_{L^1(I,H^s)} = L(\ps_n) \to 0$ and therefore $\ph_n = \ps_n(1) = \on{Fl}_1(u_n) \circ \ph \to \ph$ in the $H^s$-topology by Lemma~\ref{lem:ex_flow}.
\end{proof}

\subsection{The outer metric on $\on{Emb}(M,\R^d)$} Let $M$ be a compact manifold without boundary with $\dim M =m$.
For each $q \in \on{Emb}(M,\R^d)$ we define the seminorm $\|\cdot\|_{q,\mc O}$ on $T_q\on{Emb}(M,\R^d)$,
\[
\| u \|_{q,\mc O}^2 = \inf_{X \circ q = u} G^{\mc D}_{\on{Id}}(X,X)\,,
\]
with $u \in T_q \on{Emb}(M,\R^d)$ and the infimum taken over $X \in \mf X_c(\R^d)$. It is not difficult to see that $\|\cdot \|_{q,\mc O}$ satisfies the parallelogram law,
\[
2 \| u \|^2_{q,\mc O} + 2 \| v \|^2_{q,\mc O} = \| u + v\|^2_{q,\mc O} + \| u-v\|^2_{q,\mc O}\,,
\]
and hence defines a symmetric, positive semi-definite bilinear form
\[
G_q^{\mc O}(u,v) = \frac 14 \left( \| u + v \|^2_{q,\mc O} -  \| u - v \|^2_{q,\mc O} \right)\,.
\]
We call $G^{\mc O}$ the \emph{outer metric} induced by $G^{\mc D}$.

\begin{lemma}
\label{lem:outer_metric_pos_definite}
Let $s > d/2$. Then the  bilinear form $G^{\mc O}_q$ is positive definite.
\end{lemma}

\begin{proof}
If $u \neq 0$, then $u(x) \neq 0$ for some $x \in M$ and because $s > d/2$ we have by the Sobolev embedding theorem 
\[
|u(x)|^2 = |X(q(x))|^2 \leq C_1 \| X\|_{H^s} \leq C_2 G_{\on{Id}}^{\mc D}(X,X)
\]
for some constants $C_1,C_2 > 0$ and all $X \in \mf X_c(\R^d)$ with $X\circ q = u$. Thus $G_q^{\mc O}(u,u) \geq C_2^{-1} |u(x)|^2 > 0$ and it follows that $G_q^{\mc O}(\cdot,\cdot)$ is positive definite.
\end{proof}

\subsection{Riemannian submersions}
The group $\on{Diff}_c(\R^d)$ acts on $\on{Emb}(M,\R^d)$ from the left via
\[
(\ph, q) \mapsto \ph \circ q\,.
\]
and the inner product $G^{\mc O}$ is related to this left action in the following way: Let $q_0 \in \on{Emb}(M,\R^d)$ be fixed and consider the induced map
\[
p : \on{Diff}_{c}(\R^d) \to \on{Emb}(S^1,\R^d)\,,\quad \ph \mapsto \ph \circ q_0\,.
\]
Then $G^{\mc D}$, $G^{\mc O}$ and $p$ are related by
\[
G^{\mc O}_q(u,u) = \inf_{T_\ph p.X_\ph = u} G_{\ph}^{\mc D}(X_\ph,X_\ph)\,,
\]
provided $\ph \in \on{Diff}_{c}(\R^d)$ is such that $p(\ph) = \ph \circ q_0 = q$ and the infimum is taken over $X_\ph \in T_\ph \on{Diff}_{c}(\R^d)$. To see this note that $T_\ph p.X_\ph = X_\ph \circ q_0$ and $Y = X_\ph \circ \ph\inv \in \mf X_{c}(\R^d)$ satisfies $Y \circ q = X_\ph \circ q_0 = u$. Thus, using the right-invariance of $G^{\mc D}$,
\[
\inf_{T_\ph p.X_\ph = u} G_{\ph}^{\mc D}(X_\ph,X_\ph) = \inf_{Y \circ q = u} G_{\ph}^{\mc D}(Y \circ \ph, Y \circ \ph) = \inf_{Y \circ q = u} G_{\on{Id}}^{\mc D}(Y,Y)\,.
\]

If we ignore for now the question whether the outer metric $G^{\mc O}$ depends smoothly on $q$, we can say that 
\[
p : (\on{Diff}_c(\R^d), G^{\mc D}) \to (\on{Emb}(M,\R^d), G^{\mc O})
\]
is a \emph{Riemannian submersion}. We would like to emphasize that this is true for all choices of a base embedding $q_0$.

\subsection{The orthogonal projection}
When viewing $p$ as a Riemannian submersion, we can associate to each $\ph$ the vertical subspace
\[
\on{ver}(\ph) = \on{ker} T_\ph p = \{ X_\ph \,:\, X_\ph \circ q_0 \equiv 0 \}
\subset T_q \on{Diff}_c(\R^d) \,.
\]
It is more natural to consider the right-trivialization $\on{ver}(\ph) \circ \ph\inv$, which only depends on $q = p(\ph) = \ph \circ q_0$,
\[
\on{ver}(\ph) \circ \ph\inv = \{ Y \,:\, Y \circ q \equiv 0\} \subset \mf X_c(\R^d)\,.
\]
Let $s > d/2$. We write $\on{ver}(q)$ for the $H^s$-closure of the subspace $\on{ver}(\ph) \circ q\inv$,
\[
\on{ver}(q) = \{ X \in H^s(\R^d,\R^d) \,:\, X \circ q \equiv 0 \} =  \on{ker} \on{Tr}_q \subset H^s(\R^d,\R^d) \,.
\]
Since $\on{ker} \on{Tr}_q$ is a closed subspace of $H^s(\R^d,\R^d)$, there exists the $\|\cdot\|_A$-othogonal projection
\[
P_q : H^s(\R^d,\R^d) \to H^s(\R^d,\R^d)\,,
\]
with kernel $\on{ker} \on{Tr}_q$. To be precise, $P_q$ is characterized by the following properties,
\begin{align*}
P^2_q &= P_q\,, &
\on{ker} P_q &= \on{ker} \on{Tr}_q\,, &
\langle AP_q X, P_qY \rangle_{H^{-s}} = \langle AP_q X, Y \rangle_{H^{-s}}\,.
\end{align*}
and the last identity can be rewritten as
\[
P_q^\ast A = P_q^\ast A P_q = A P_q\,.
\]

\subsection{Order of the outer metric}
We have seen in Lemma~\ref{lem:outer_metric_pos_definite} that the outer metric $G^{\mc O}$ is positive definite. Using the trace and extension operators from Lemma~\ref{lem:trace_theorem} we can also determine the topology induced by each $G^{\mc O}_q$.

\begin{lemma*} Let $s > d/2$, $s' = s - (d-m)/2$ and $q \in \on{Emb}(M,\R^d)$. The outer metric $\| \cdot \|_{q,\mc O}$ induces the Sobolev $H^{s'}$-topology on $T_q \on{Emb}(M,\R^d) \cong C^\infty(M,\R^d)$.
\end{lemma*}

\begin{proof}
Using the orthogonal projection we can write
\[
\| u \|_{q,\mc O} = \| P_q \on{Ex}_q u \|_A\,,
\]
where $\on{Ex}_q$ is any extension map, i.e., a bounded right inverse of the trace map as in Lemma~\ref{lem:trace_theorem}. Then
\[
\| u \|_{q,\mc O} \leq \| \on{Ex}_q u \|_A 
\leq C_1 \| \on{Ex}_q u \|_{H^s(\R^d)} \leq C_2 \| u \|_{H^{s'}(M)}\,,
\]
with some constants $C_1,C_2$. On the other hand every $u \in H^{s'}(M,\R^d)$ can be written as $u = \on{Tr}_q \on{Ex}_q u$ and because $\on{Ex}_q u - P_q \on{Ex}_q u \in \on{ker} \on{Tr}_q$ we have
\[
u = \on{Tr}_q P_q \on{Ex}_q u\,.
\]
Thus
\begin{align*}
\| u \|_{H^{s'}(M)} &= \| \on{Tr}_q P_q \on{Ex}_q u \|_{H^{s'}(M)} \\
&\leq C_1 \| P_q \on{Ex}_q u \|_{H^s(\R^d)} \leq 
C_2 \| P_q \on{Ex}_q u \|_{A} = C_2 \| u \|_{q,\mc O}\,,
\end{align*}
again with some constants $C_1,C_2$. This shows that for all $q \in \on{Emb}(M,\R^d)$, the outer metric $\| \cdot \|_{q,\mc O}$ induces the Sobolev $H^{s'}$-topology on $T_q \on{Emb}(M,\R^d)$.
\end{proof}

\subsection{Notation and assumptions} To study the continuous dependance of $G^{\mc O}_q$ on the basepoint $q$, we introduce the following notation and assumptions which will remain valid until the end of the section. 

Let $s > d/2$, $s' = s-(d-m)/2$ and $q \in \on{Emb}(M,\R^d)$. Denote the Riesz isomorphism of $G^{\mc O}_q(\cdot, \cdot)$ on $H^{s'}(M,\R^d)$ by
\[
A_q : H^{s'}(M,\R^d) \to H^{-s'}(M,\R^d)\,,
\]
and its inverse by $B_q = A_q\inv$,
\[
B_q : H^{-s'}(M,\R^d) \to H^{s'}(M,\R^d)\,.
\]
Furthermore,
\[
P_q : H^s(\R^d,\R^d) \to H^s(\R^d,\R^d)\,,
\]
will denote the $\| \cdot \|_{A}$-orthogonal projection as defined above.

\subsection{Continuity of the outer metric $G^{\mc O}$ on $\on{Emb}(M,\R^d)$.}
The goal of the following lemmas is to show that the Riemannian metric $G^{\mc O}$, defined by the formula
\[
G^{\mc O}_q(u,u) = \inf_{T_\ph p.X_\ph = u} G_{\ph}^{\mc D}(X_\ph,X_\ph)\,,
\]
does depend continuously on the basepoint $q$. The proof does rely on the fact that the inner product $G^{\mc D}_{\on{Id}}(\cdot,\cdot)$ is a Sobolev $H^s$-inner product and that we have trace and extension theorems for Sobolev spaces available to us. The following example shows that continuous dependence of the induced metric on the basepoint is not automatic.

\begin{example}
Let $I = [-1,1]$ be a closed interval in $\R$ and $M = \{ \ast \}$ a $0$-dimensional manifold, such that $\on{Emb}(M,\R) \cong \R$. Consider the following inner product on $\mf X_c(\R)$,
\[
G^{\mc D}_{\on{Id}}(X,Y) = \int_\R X(y)Y(y) + \one_I(y) X'(y) Y'(y) \ud y\,,
\]
and the corresponding right-invariant Riemannian metric $G^{\mc D}$ on $\on{Diff}_c(\R)$. This metric behaves like an $L^2$-metric on $I^c$ and like an $H^1$-metric on $I$. What is the induced metric $G^{\mc O}$ on $\R$? When $x \in I^c$, then
\[
G^{\mc O}_x(1,1) = 0\,,\quad |x| > 1\,,
\]
because we can find vector fields $X$ satisfying $X(x) = 1$, whose support is contained in $I^c$ and with arbitrary small $L^2$-norm. For $x \in I$, however, we have the Sobolev embedding $H^1(I) \hookrightarrow C(I)$, and hence
\[
|X(x)|^2 \leq C \|X\|^2_{H^1(I)} = C \int_I X(y)^2 + X'(y)^2 \ud y \leq C G^{\mc D}_{\on{Id}}(X,X) \,,
\]
for some $C > 0$. By taking the infimum over all $X$ such that $X(x) = 1$, we obtain
\[
G^{\mc O}_x(1,1) \geq C\inv\,,\quad |x|\leq 1\,.
\]
Hence the induced outer metric $G^{\mc O}$ is not continuous at $|x|=1$.
\end{example}

The first step in the proof is to relate the Riesz isomorphism $A_q$, its inverse $B_q$ and the orthogonal projection $P_q$ with the trace and extension operators for Sobolev spaces.

\begin{proposition}
\label{prop:ABP_formulas}
The following relations hold for $q \in \on{Emb}(M,\R^d)$,
\begin{align*}
A_q &= \on{Ex}_q^\ast P_q^\ast A P_q \on{Ex}_q\,, &
B_q &= \on{Tr}_q A\inv \on{Tr}_q^\ast\,, &
P_q &= A\inv \on{Tr}_q^\ast A_q \on{Tr}_q\,.
\end{align*}
\end{proposition}

Note that the formulas do not depend on the precise choice of the extension map $\on{Ex}_q$. The only requirement is that $\on{Tr}_q \on{Ex}_q = \on{Id}_{H^{s'}(M)}$.

\begin{proof}
Let $u \in H^{s'}(M,\R^d)$. The operator $A_q$ is defined via
\[
\langle A_q u, u \rangle_{H^{-s'}(M)} = \inf_{X \circ c = u} \langle AX, X \rangle_{H^{-s}(\R^d)}\,,
\]
and the infimum is attained for $X = P_q \on{Ex}_q u$. Thus
\begin{align*}
\langle A_q u, u \rangle_{H^{-s'}(M)}
&= \langle A P_q \on{Ex}_q u, P_q \on{Ex}_q u \rangle_{H^{-s}(\R^d)} \\
&= \langle \on{Ex}_q^\ast P_q^\ast A P_q \on{Ex}_q u, u \rangle_{H^{-s'}(M)}\,.
\end{align*}
Next we verify the formula for $B_q$. Let $B_q$ be defined as above. First note that because $X - \on{Ex}_q \on{Tr}_q X \in \on{ker} \on{Tr}_q$, it follows that
\[
P_q X = P_q \on{Ex}_q \on{Tr}_q X\,,\quad
\forall X \in H^s(\R^d,\R^d)\,.
\]
Similarly, because $X - P_q X \in \on{ker} \on{Tr}_q$, we obtain
\[
\on{Tr_q} X = \on{Tr}_q P_q X\,,\quad
\forall X \in H^s(\R^d,\R^d)\,.
\]
Therefore, using the identity $P_q^\ast A P_q = P_q^\ast A$,
\begin{align*}
A_q B_q 
&= \on{Ex}_q^\ast P_q^\ast A P_q \on{Ex}_q \on{Tr}_q A\inv \on{Tr}_q^\ast \\
&= \on{Ex}_q^\ast P_q^\ast A P_q A\inv \on{Tr}_q^\ast \\
&= \on{Ex}_q^\ast P_q^\ast \on{Tr}_q^\ast
= (\on{Tr}_q \on{Ex}_q)^\ast = \on{Id}_{H^{-s'}(M)}\,.
\end{align*}
Similarly, using $P_q^\ast A P_q = A P_q$, we obtain
\begin{align*}
B_q A_q 
&= \on{Tr}_q A\inv \on{Tr}_q^\ast \on{Ex}_q^\ast P_q^\ast A P_q \on{Ex}_q  \\
&= \on{Tr}_q A\inv P_q^\ast A P_q \on{Ex}_q \\
&= \on{Tr}_q P_q \on{Ex}_q = \on{Tr}_q \on{Ex}_q = \on{Id}_{H^{s'}(M)}\,.
\end{align*}
Finally, the calculation
\begin{align*}
A\inv \on{Tr}_q^\ast A_q \on{Tr}_q
&= A\inv \on{Tr}_q^\ast \on{Ex}_q^\ast P_q^\ast A P_q \on{Ex}_q \on{Tr}_q \\
&= A\inv (P_q \on{Ex}_q \on{Tr}_q)^\ast A (P_q \on{Ex}_q \on{Tr}_q) \\
&= A\inv P_q^\ast A P_q = P_q\,,
\end{align*}
proves the formula for $P_q$.
\end{proof}

The formulas derived in Proposition~\ref{prop:ABP_formulas} allow us prove continuity of these maps. Note that we prove the joint continuity of the maps
\[
A :  \on{Emb}(M,\R^d) \x H^{s'} \to H^{-s'}
\]
and not the stronger statement that
\[
A : \on{Emb}(M,\R^d) \to L(H^{s'}, H^{-s'})\,,
\]
is continuous. In fact, we do not expect the latter to be true.

\begin{proposition}
\label{prop:ABP_continuous}
The maps
\begin{align*}
A &: \on{Emb}(M,\R^d) \x H^{s'}(M,\R^d) \to H^{-s'}(M,\R^d) \\
B &: \on{Emb}(M,\R^d) \x H^{-s'}(M,\R^d) \to H^{s'}(M,\R^d) \\
P &: \on{Emb}(M,\R^d) \x H^{s}(\R^d,\R^d) \to H^{s}(\R^d,\R^d)
\end{align*}
are continuous.
\end{proposition}

\begin{proof}
The continuity of $B$ follows from the formula
\[
B_q = \on{Tr}_q A\inv \on{Tr}_q^\ast
\]
and Proposition~\ref{prop:trace_theorem_cont}. We will show continuity of $A$, using $A_q = B_q\inv$ and by applying Lemma~\ref{lem:cont_inverse_param}. To do so we need to show that $\| A_q \|_{L(H^{s'},H^{-s'})}$ is locally bounded. The maps $\on{Ex}$ and $\on{Ex^\ast}$ are continuous by Lemma~\ref{prop:trace_theorem_cont} and therefore their operator norms are locally bounded. The map $P_q$ is an orthogonal projection and therefore $\| P_q X \|_A \leq \| X \|$. Because the norm $\| \cdot \|_A$ is equivalent to the $H^s(\R^d)$-norm, $\| P_q \|_{L(H^s,H^s)}$ is locally bounded. Thus $\| A_q \|_{L(H^{s'},H^{-s'})}$ is locally bounded and hence $A$ is continuous. Finally, the formula
\[
P_q = A\inv \on{Tr}^\ast_q A_q \on{Tr}_q\,,
\]
shows that $P$ is continuous as well.
\end{proof}

The following lemma was used to show continuity of $A_q = B_q\inv$ using the continuity of $B_q$. 
\begin{lemma}
\label{lem:cont_inverse_param}
Let $U$ be a metrizable topological space, $E, F$ Banach spaces and
\[
A : U \x E \to F
\]
a continuous map, such that $A_x \in GL(E,F)$ for all $x \in U$. If $\| A\inv_x \|_{L(F,E)}$ is locally bounded, then
\[
A\inv : U \x F \to E
\]
is continuous.
\end{lemma}

\begin{proof}
Let $(x_n, z_n) \to (x,z)$ in $U \x F$ and write
\[
A\inv_{x_n} z_n -A\inv_x z = 
A\inv_{x_n} (z_n - z) + A\inv_{x_n} (A_x - A_{x_n}) A\inv_x z\,.
\]
The convergence $A\inv_{x_n} z_n \to A\inv_{x} z$ now follows from the boundedness of $\| A\inv_{x_n} \|_{L(F,E)}$ and the convergence $(A_x - A_{x_n}) A\inv_x z \to 0$.
\end{proof}

With the help of Proposition~\ref{prop:ABP_continuous} we are able to show the continuity of the outer metric.

\begin{corollary}
\label{cor:X_map}
The outer Riemannian metric $G^{\mc O}$ is continuous as a map
\[
G^{\mc O} : \on{Emb}(M,\R^d) \x H^{s'}(M,\R^d) \x H^{s'}(M,\R^d) \to \R\,.
\]
and the map $X : T\on{Emb}(M,\R^d) \to H^s(\R^d,\R^d)$, given by
\[
(q,u) \mapsto X(q,u) = P_q \on{Ex}_q u \,,
\]
is continuous and satisfies
\[
G^{\mc O}_q(u,u) = G_{\on{Id}}^{\mc D}(X(q,u), X(q,u))\,.
\]
\end{corollary}

\begin{proof}
The continuity of $G^{\mc O}$ follows directly from
\[
G^{\mc O}_q(u,v) = \langle A_q u, v \rangle_{H^{s'}}\,.
\]
It is clear that $X = P \circ \on{Ex}$ is continuous. To show the identity connecting $G^{\mc O}$ and $G^{\mc D}$ we calculate
\begin{align*}
G^{\mc D}_{\on{Id}}(X(q,u),X(q,u)) 
&= \langle AP_q\on{Ex}_q u, P_q\on{Ex}_q u \rangle_{H^{-s}(\R^d)} \\
&= \langle \on{Ex}_q^\ast P_q^\ast AP_q\on{Ex}_q u, u\rangle_{H^{-s'(M)}} \\
&= \langle A_q u, u \rangle_{H^{-s'}(M)} = G^{\mc O}_q(u,u)\,.\qedhere
\end{align*}
\end{proof}

\subsection{The geodesic distance and orbits of the $\on{Diff}_c(\R^d)$-action.}

If $M, N$ are finite-dimesnional manifolds and $p : (M,\ga_M) \to (N,\ga_N)$ is a Riemannian submersion, then the geodesic distance functions are related by
\begin{equation}
\label{eq:submersion_infimum_orbits}
\on{dist}_N(y_1,y_2) = \inf_{p(x_2)=y_2} \on{dist}_M(x_1,x_2)\,,
\end{equation}
provided $p(x_1) = y_1$. The proof proceeds by considering paths connecting $y_1$ and $y_2$ in $N$ and lifting them horizontally to paths $M$ connecting $x_1$ and some point $x_2$ in the preimage $p\inv(y_2)$. In our situation, we are not able to lift paths horizontally.

\begin{example}
For $s \in \mb N$, consider a right-invariant Riemannian metric $G^{\mc D}$ on $\on{Diff}_c(\R^d)$ of the form
\[
G^{\mc D}_{\on{Id}}(X,Y) = \int_{\R^d} \langle LX, Y \rangle \ud x\,,
\]
where $L$ is a positive, symmetric, elliptic differential operator of order $2s$, e.g., $L=(\on{Id} - \De)^{s}$. For a given $q_0 \in \on{Emb}(M,\R^d)$ we consider the vertical space of the submersion $p(\ps) = \ps \circ q_0$ at $q \in \on{Emb}(M,\R^d)$. If $p(\ph) = q$, the right-trivialization of $\on{ver}(\ph) = \ker T_\ph p$ is
\[
\on{ver}(\ph) \circ \ph\inv = \{ Y \,:\, Y \circ q \equiv 0 \} \subseteq \mf X_c(\R^d)\,.
\]
What is the $G^{\mc D}$-horizontal complement of this subspace? Assume $X \in \mf X_c(\R^d)$ is such that
\[
G^{\mc D}_{\on{Id}}(X,Y) = \int_{\R^d} \langle LX, Y \rangle \ud x = 0\,,
\]
for all $Y \in \on{ver}(\ph) \circ \ph\inv$. Because we only require $Y$ to vanish along a submanifold this implies $LX = 0$ and because $L$ is a positive elliptic differential operator, this means $X = 0$. Hence the $G^{\mc D}$-orthogonal complement in $\mf X_c(\R^d)$ is trivial. We need to consider the space of $H^s(\R^d,\R^d)$ of $H^s$-vector fields to obtain a nontrivial orthogonal complement. However, the lack of an orthognal complement in the smooth category means that for this class of metrics no path in $\on{Emb}(M,\R^d)$, apart from the constant one, can be lifted horizontally to a path in $\on{Diff}_c(\R^d)$. 
\end{example}

As the above example shows, we should not attempt to generalize the finite-dimensional proof to prove the identity~\eqref{eq:submersion_infimum_orbits}. Instead we have to proceed by hand, utilizing the group structure of the diffeomorphism group and the existence of smoothing operators.

\begin{proposition}
\label{prop:infimum_orbits}
Let $q_1, q_2 \in \on{Emb}(M,\R^d)$ be in the same connected component. Then
\[
\on{dist}^{\mc O}(q_1, q_2) = \inf_{q_2 = \ph \circ q_1} \on{dist}^{\mc D}(\on{Id}_{\R^d}, \ph)\,,
\]
where the infimum is taken over $\ph$ in $\on{Diff}_c(\R^d)$.
\end{proposition}

\begin{proof}
It is clear that the inequality
\[
\on{dist}^{\mc O}(q_1, q_2) \leq \on{dist}^{\mc D}(\on{Id}, \ph)\,,
\]
holds for all $\ph \in \on{Diff}_c(\R^d)$ with $q_2 = \ph \circ q_1$, since any path $\ps(t)$ in $\on{Diff}_c(\R^d)$ can be projected to a path in $\on{Emb}(M,\R^d)$ via $q(t) = p \circ \ps(t)$ and we have $\on{Len}^{\mc O}(q) \leq \on{Len}^{\mc D}(\ps)$. Thus it remains to construct a sequence $(\ph_k)_{k \in \mathbb N}$ in $\on{Diff}_c(\R^d)$ with
$\on{dist}^{\mc D}(\on{Id}, \ph_k) \to \on{dist}^{\mc O}(q_1, q_2)$.

Given $\ep > 0$, choose a smooth path $q = q(t)$ in $\on{Emb}(M,\R^d)$ between $q_1$, $q_2$ with $\on{Len}^{\mc O}(q) < \on{dist}^{\mc O}(q_1,q_2) + \ep$. Let $X(t) = X(q(t), q_t(t))$ be a continuous path in $H^s(\R^d,\R^d)$, obtained by applying Corollary~\ref{cor:X_map}. Denote its flow by $\ph(t) \in \mc D^s(\R^d)$. Then 
\[
\on{dist}^{\mc D}(\on{Id}, \ph(1)) \leq \on{Len}^{\mc D}(\ph) = \on{Len}^{\mc O}(c)
\leq \on{dist}^{\mc O}(q_1,q_2) + \ep\,.
\]
Applying the smoothing operator $S_k$ from Lemma~\ref{ss:smoothing_operator} to $X$ we obtain a continuous path $A_k X(t)$ in $C^\infty_c(\R^d,\R^d)$ and we denote its flow by $\ph_k(t) \in \on{Diff}_c(\R^d)$. Using Cororllary~\ref{ss:smoothing_operator} and the $\mc D^s(\R^d)$-continuity of the flow map from Lemma~\ref{lem:ex_flow}, we obtain $\ps_k := \ph_k(1) \to \ph(1)$ in $H^s$ as $k \to \infty$ and $\ps_k \in \on{Diff}_c(\R^d)$. It follows from the trace theorem, Lemma~\ref{lem:trace_theorem}, that $\ps_k \circ q_1 \to \ph(1) \circ q_1 = q_2$ in $H^{s'}(M,\R^d)$. Using the extension operator from Lemma~\ref{lem:trace_theorem} we define
\[
\et_k = \on{Id}_{\R^d} + \on{Ex}_{q_2}(\ps_k \circ q_1 - q_2)\,.
\]
Then $\et_k \circ q_2 = \ps_k \circ q_1$ and if $\ps_k \circ q_1$ is sufficiently close to $q_2$, then $\et_k \in \mc D^s(\R^d)$; furthermore, since $\ps_k \circ q_1 -q_2 \in C^\infty(M,\R^d)$, we have in fact $\et_k \in C^\infty_c(\R^d,\R^d)$ and hence $\et_k \in \on{Diff}_c(\R^d)$. Let $\ph_k := \et_k \inv \circ \ps_k \in \on{Diff}_c(\R^d)$. Then $\ph_k \circ q_1 = \et_k\inv \circ \ps_k \circ q_1 = q_2$. We have $\ph_k \to \ph(1)$ in $H^s$ and in particular 
$\on{dist}^{\mc D}(\on{Id}, \ph_k) < \on{dist}^{\mc D}(\on{Id}, \ph(1)) + \ep$ for $k$ large enough. We obtain therefore the inequality
\[
\on{dist}^{\mc D}(\on{Id}, \ph_k) \leq \on{dist}^{\mc O}(q_1,q_2) + 2\ep\,,
\]
and the statement of the lemma follows by taking the infimum over all $\ep$.
\end{proof}

\subsection{Smoothing operators}
\label{ss:smoothing_operator}
The next lemma, used in the proof of Proposition~\ref{prop:infimum_orbits}, is a more detailed way of saying that $C^\infty_c(\R^d)$ is dense in $H^s(\R^d)$. The proof follows~\cite[p.~226]{Treves1975}. Choose $\et \in C^\infty_c(\R^d)$ with $\et \equiv 1$ for $|x| < 1$ and define the functions $\et_k(x) = \et(k\inv x)$ and $\ch(\xi) = \one_{\{|\xi| \leq k \}}(\xi)$. With this we define the operators
\[
S_k f(x) = \et_k(x).\ch_k(D) f(x)\,,
\]
where $\ch_k(D)$ is the Fourier multiplier $\ch_k(D) f = ( \ch_k. \wh f )^\vee$.

\begin{lemma*}
The operators $S_k : C^\infty_c(\R^d) \to C^\infty_c(\R^d)$ are continuous with the following properties for each $s \geq 0$ :
\begin{enumerate}
\item
\label{Ak_smoothing_op_extend}
They extend continuously to $S_k : H^s(\R^d) \to C_c^\infty(\R^d)$.
\item
\label{Ak_smoothing_op_converge}
For each $f \in H^s$, $S_k f \to f$ in $H^s$ as $k \to \infty$.
\item
As operators $S_k : H^s \to H^s$, the family $\{ S_k \}$ is uniformly bounded.
\end{enumerate}
\end{lemma*}

\begin{proof}
Let $k \in \mb N$ be fixed. The operator $\ch_k(D) : H^s \to H^{s+m}$ is continuous for all $m \geq 0$ and hence so is $\ch_k(D) : H^s \to H^\infty$. The multiplication operator $\et_k : H^\infty \to C^\infty_c(B_k)$ with $B_k = \on{supp} \et_k$ is continuous and thus $\et_k : H^\infty \to C^\infty_c(\R^d)$ is continuous as well. Therefore, both $S_k : C_c^\infty \to C^\infty_c$ and the extension $S_k : H^s \to C^\infty_c$ are continuous operators.

We can estimate the operator norms of $\ch_k(D)$ and $\et_k$ via
\begin{align*}
\| \ch_k(D) f \|_{H^s} &= \bigg(\int_{|\xi| \leq k} \langle \xi \rangle^{2s} |\wh f(\xi)|^2 \ud \xi \bigg)^{1/2}
\leq \| f \|_{H^s} \\
\| \et_k f \|_{H^s} & \leq 2^s \| \langle \xi \rangle^s \wh \et_k \|_{L^1} \| f \|_{H^s}
= 2^s \| \langle k\inv \xi \rangle^s \wh \et \|_{L^1} \| f \|_{H^s}
\leq 2^s \| \langle \xi \rangle^s \wh \et \|_{L^1} \| f \|_{H^s}\,,
\end{align*}
where we use Lemma~\ref{lem:sob_mult_estimate} and that $\wh \et_k(\xi) = k^d \wh \et(k\xi)$.
This shows that as operators $H^s \to H^s$ the family $\{ S_k \}$ is uniformly bounded. Thus it is sufficient to verify the convergence $S_k f \to f$ for $f$ from a dense subset of $H^s$. It is easy to see that
$\ch_k(D) g \to g$ and $\et_k g \to g$ in $H^s$ when $g \in \schwartz$.
\end{proof}

\begin{corollary*}
Let $s \geq 0$, $f \in L^1(I,H^s(\R^d))$ and define $(S_k f)(t) = S_k(f(t))$. Then $S_k f \to f$ in $L^1(I,H^s(\R^d))$.
\end{corollary*}

\begin{proof}
Because the operators $S_k$ are uniformly bounded, we can use the theorem of dominated convergence to show that
\[
\lim_{k \to \infty} \| S_k f - f \|_{L^1(I,H^s)}
= \int_I \lim_{k \to \infty} \| S_k f(t) - f(t) \|_{H^s} \ud t = 0\,.\qedhere
\]
\end{proof}

\subsection{Topology induced by $\on{dist}^{\mc O}$ on $\on{Emb}(M,\R^d)$.}
The next two lemmas identify the topology induced by the geodesic distance of an outer metric on the space of embeddings. We show that for $s > d/2+1$ the topology coindices with the $H^{s'}$-topology, where $s' = s-(d-m)/2$. In fact Propositions~\ref{prop:flat_outer_bound} and~\ref{prop:outer_flat_bound} show that the identity map
\[
\on{Id} : (\on{Emb}(M,\R^d), \on{dist}^{\mc O}) \to (\on{Emb}(M,\R^d), \|\cdot\|_{H^{s'}})\,,
\]
is locally bi-Lipschitz. There is a slight asymmetry here, because the above map is not just locally Lipschitz but Lipschitz on every metric ball while the inverse map is simply locally Lipschitz.

\begin{proposition}
\label{prop:flat_outer_bound}
Let $s > d/2+1$ and $s' = s-(d-m)/2$. Given $q_0 \in \on{Emb}(M,\R^d)$ and $R > 0$, there exists a constant $C = C(q_0, R) > 0$ such that
\[
\| q_1 - q_2 \|_{H^{s'}} \leq C \on{dist}^{\mc O}(q_1, q_2)\,,
\]
for all $q_1, q_2 \in B^{\mc O}(q_0, R)$.
\end{proposition}

\begin{proof}
Let $\ph_1, \ph_2 \in \on{Diff}_c(\R^d)$ be such that $\ph_1 \circ q_0 = q_1$, $\ph_2 \circ q_0 = q_2$ and $\on{dist}^{\mc D}(\on{Id}_{\R^d},\ph_i) < R$. Then we can estimate
\[
\| q_1 - q_2 \|_{H^{s'}} = \| \ph_1 \circ q_0 - \ph_2 \circ q_0 \|_{H^{s'}}
= \left\| \on{Tr}_{q_0} \left(\ph_1 - \ph_2\right) \right\|_{H^{s'}}
\leq C_1 \| \ph_1 - \ph_2 \|_{H^s}\,,
\]
with some constant $C_1$ via Lemma~\ref{lem:trace_theorem}. Using \cite[Lemma~6.6]{Bruveris2014_preprint} we obtain
\[
\| \ph_1 - \ph_2 \|_{H^s} \leq C_2' \on{dist}^{\mc D}(\ph_1, \ph_2)\,,
\]
for some constant $C_2'$ and hence
\[
\| q_1 - q_2 \|_{H^{s'}} \leq C_2 \on{dist}^{\mc D}(\on{Id}_{\R^d}, \ph_2 \circ \ph_1\inv)\,.
\]
By noting that $\ph_2 \circ \ph_1\inv \circ q_1 = q_2$ and taking the infimum over all $\ph_1, \ph_2$ we arrive, using Proposition~\ref{prop:infimum_orbits}, at
\[
\| q_1 - q_2 \|_{H^{s'}} \leq C_2 \on{dist}^{\mc O}(q_1,q_2)\,.\qedhere
\]
\end{proof}

\begin{proposition}
\label{prop:outer_flat_bound}
Let $s > d/2+1$ and $s' = s-(d-m)/2$. Given $q_0 \in \on{Emb}(M,\R^d)$, there exists $R = R(q_0)> 0$ and $C = C(q_0) > 0$ such that
\[
\on{dist}^{\mc O}(q_1,q_2) \leq C \| q_1 - q_2 \|_{H^{s'}}\,,
\]
for all $\| q_i - q_0\|_{H^{s'}} < R$, $i=1,2$.
\end{proposition}

\begin{proof}
Given $q_0$, consider the map
\[
q \mapsto \ph = \on{Id}_{\R^d} + \on{Ex}_{q_0}(q-q_0)\,.
\]
It satisfies $\ph \circ q_0 = q_1$. If $\|q - q_0\|_{H^{s'}} < R$ with $R$ sufficiently small, then $\ph \in \mc D^{s}(\R^d)$ and all such $\ph$ satisfy the assumptions of Lemma~\ref{lem:diff_ball_bound}; furthermore, if $q$ is smooth, then $\ph \in \on{Diff}_c(\R^d)$. We choose such an $R$.

Let $q_1, q_2$ be given and define $q(t) = (1-t)q_1 + tq_2$ as well as 
\[
\ph(t) = \on{Id}_{\R^d} + \on{Ex}_{q_0}((1-t)q_1 + tq_2-q_0)\,.
\]
Then $\ph(t) \circ q_0 = q(t)$ and $\ph(t)$ satisfies the assumptions of Lemma~\ref{lem:diff_ball_bound} for all $t \in [0,1]$. Then,
\[
\on{dist}^{\mc O}(q_1,q_2) \leq \on{dist}^{\mc D}(\ph_1,\ph_2)
\leq \int_0^1 \| \p_t \ph(t) \circ \ph(t)\inv \|_{H^s} \ud t
\leq C_1 \int_0^1 \| \p_t \ph(t) \|_{H^s} \ud t\,,
\]
with the constant $C_1$ obtained from Lemma~\ref{lem:comp_diff_ball}. Because $\p_t \ph(t) = \on{Ex}_{q_0}(q_2-q_1)$, we also have the estimate
\[
\| \p_t \ph(t) \|_{H^s} = \| \on{Ex}_{q_0}(q_2-q_1) \|_{H^s} \leq
C_2 \| q_2 - q_1 \|_{H^{s'}}\,.
\]
Putting everything together we arrive at
\[
\on{dist}^{\mc O}(q_1,q_2) \leq C \| q_2 - q_1 \|_{H^{s'}}\,,
\]
for some constant $C$.
\end{proof}

It seems tempting to combine Propositions~\ref{prop:flat_outer_bound} and~\ref{prop:outer_flat_bound} to show that the metric completion of $(\on{Emb}(M,\R^d), \on{dist}^{\mc O})$ is the space
\[
\mc E^{s'}(M,\R^d) = \{ q \in H^{s'}(M,\R^d) \,:\, q\text{ is a $C^1$-embedding} \}\,,
\]
of $H^{s'}$-embeddings as the local equivalence between the geodesic distance $\on{dist}^{\mc O}$ and the norm $\|\cdot\|_{H^{s'}}$ would suggest. However, the equivalence is only local. In Proposition~\ref{prop:flat_outer_bound} it holds on arbitrary metric balls but in Proposition~\ref{prop:outer_flat_bound} the estimate is truly local and the radius $R(q_0)$ depends on the embedding. 

To prove the statement about the metric completion of $\on{Emb}(M,\R^d)$ it would be enough to strengthen Proposition~\ref{prop:outer_flat_bound} by allowing $q_0$ to be any emedding in $\mc E^{s'}(M,\R^d)$; now $q_0$ has to be $C^\infty$-smooth. However, the proof of Proposition~\ref{prop:outer_flat_bound} uses the extension map $\on{Ex}_{q_0}$, which to our knowledge only has been studied for smooth embeddings $q_0$. Thus we leave the following conjecture.

\begin{conjecture}
Let $s > d/2+1$ and $s' = s - (d-m)/2$. Then the metric completion of $(\on{Emb}(M,\R^d), \on{dist}^{\mc O})$ is the space $\mc E^{s'}(M,\R^d)$ of $H^{s'}$-embeddings. 
\end{conjecture}

As argued above this conjecture is related to the following conjecture about extending the trace and extension operators to suitable spaces of Sobolev embeddings.

\begin{conjecture}
Let $M$ be a smooth, compact manifold without boundary with $\dim M = m$. Let $s > d/2$ and $s'=s-(d-m)/2$. Then the trace operator
\[
\on{Tr} : \mc E^{s'}(M,\R^d) \x H^s(\R^d)  \to H^{s'}(M)\,,\quad (q,f) \mapsto \on{Tr}_q f\,,
\]
is continuous and around each $q \in \mc E^{s'}(M,\R^d)$ there exists an open neighborhood $\mc U \subseteq \mc E^{s'}(M,\R^d)$ and a continuous extension map
\[
\on{Ex} : \mc U \x H^{s'}(M)  \to H^{s}(\R^d)\,,\quad (q,f) \mapsto \on{Ex}_q f\,,
\]
satisfying $\on{Tr}_q \on{Ex}_q f = f$.
\end{conjecture}

\section{%
\texorpdfstring%
{Comparing inner and outer metrics on $\on{Emb}(S^1,\R^d)$}%
{Comparing inner and outer metrics on Emb(S1,Rd)}%
}
\label{sec:comparison}

The other class of Riemannian metrics we will consider are intrinsically defined Sobolev metrics or short $\emph{inner metrics}$. Here `inner' refers to the fact that we do not use deformations of the ambient space to define the metric. We will work only with $M=S^1$ in this section, because the theory of inner metrics is not sufficiently developed for higher-dimensional manifolds to prove the comparison theorem of interest.

As opposed to outer metrics considered in the previous section that are defined on the space of embeddings, inner metrics are naturally defined on the space of immersions,
\[
\on{Imm}(S^1,\R^d) = \{ c \in C^\infty(S^1,\R^d) \,:\, c'(\th) \neq 0,\, \forall\th \in S^1 \}\,.
\]
This is because inner metrics do not take into account the global geometry of the curve and hence self-intersections pose no problem to these metrics.

\begin{definition}
Let $n \in \mathbb N$. A Riemannian metric $G^{\mc I}$ on $\on{Imm}(S^1,\R^d)$ of the form
\[
G_c^{\mc I}(u,v) = \int_{S^1} a_0 \langle u, v \rangle + a_1 \langle D_s u, D_s v \rangle + \dots +
a_n \langle D_s^n u, D_s^n v \rangle \ud s\,,
\]
with constants $a_0, a_n > 0$ and $a_j \geq 0$ is called a \emph{Sobolev metric with constant coefficients} of order $n$. In the above equation $D_s u = \frac{1}{|c'|} u'$ denotes differentiation with respect to arc length and $\ud s = |c'| \ud \th$ integration with respect to arc length.
\end{definition}

When $n \geq 2$, the metric $G^{\mc I}$ extends to a smooth Riemannian metric on the Sobolev completion
\[
\mc I^n(S^1,\R^d) = \{ c \in H^n(S^1,\R^d) \,:\, c'(\th) \neq 0,\;\forall \th \in S^1\}\,.
\]
It is shown in \cite{Bruveris2014,Bruveris2015} that the Riemannian manifold $(\mc I^n(S^1,\R^d), G^{\mc I})$ is both geodesically and metrically complete. The proof relies on the following crucial estimate, which we will rely on later.

\begin{lemma}
\label{lem:curve_norm_equiv}
Let $n \geq 2$. Given any $M > 0$ and $C > 0$, there exists a constant $C_n = C_n(M,C)$, such that for all $c \in \mc I^n(S^1,\R^d)$ with
\[
\inf_{\th \in S^1} |c'(\th)| > M\;\;\text{and}\;\;
\| c \|_{H^n(d\th)} < C\,,
\]
and all $u \in H^n(S^1,\R^d)$,
\[
C_n\inv \|u \|_{H^n(d\th)} \leq \sqrt{G^{\mc I}_c(u,u)} \leq C_n \|u\|_{H^n(d\th)}\,.
\]
\end{lemma}

\subsection{The main theorem}
\label{thm:inner_outer_bound}
We now come to the main theorem, comparing inner and outer metrics. We assume that we have an inner metric $G^{\mc I}$ of order $n$ on $\on{Imm}(S^1,\R^d)$ and a right-invariant Sobolev metric $G^{\mc D}$ of order $s$ on $\on{Diff}_c(\R^d)$. The latter induces an outer metric $G^{\mc O}$ of order $s'=s-(d-1)/2$ on $\on{Emb}(S^1,\R^d)$. We now restrict the inner metric from $\on{Imm}(S^1,\R^d)$ to $\on{Emb}(S^1,\R^d)$, which is an open subset of $\on{Imm}(S^1,\R^d)$. If $s'\geq n$, then we can bound the $G^{\mc I}$-geodesic distance by the $G^{\mc O}$-geodesic distance on arbitrary $G^{\mc O}$-metric balls.

\begin{theorem*}
Let $n \geq 2$ and $s \geq n + (d-1)/2$ be the orders of the metrics $G^{\mc I}$ and $G^{\mc D}$ respectively and denote by $G^{\mc O}$ the metric on $\on{Emb}(S^1,\R^d)$ induced by $G^{\mc D}$. Then, given $c_0 \in \on{Emb}(S^1,\R^d)$ and $R>0$, there exists $C = C(c_0, R)$, such that
\[
\on{dist}^{\mc I}(c_1, c_2) \leq C \on{dist}^{\mc O}(c_1, c_2)\,,
\]
holds for all $c_1, c_2 \in B^{\mc O}(c_0, R)$.
\end{theorem*}

Here we use the notation
\[
B^{\mc O}(c,R) = \{ c_1 \in \on{Emb}(S^1,\R^d) \,:\, \on{dist}^{\mc O}(c_1,c) < R\}\,,
\]
and similarly $B^{\mc D}(\ph,R)$ for $G^{\mc D}$-metric balls on $\on{Diff}_c(\R^d)$.

\begin{proof}
Take $c_1, c_2 \in B^{\mc O}(c_0,R)$. Since $\on{dist}^{\mc O}(c_0, c_1) < R$, there exists by Lemma~\ref{prop:infimum_orbits} a diffeomorphism $\ps \in \on{Diff}_{c}(\R^d)$ with $c_1 = \ps \circ c_0$ and $\on{dist}^{\mc D}(\on{Id}_{\R^d}, \ps) < R$. Furthermore, we choose using the same lemma a smooth path $\ph = \ph(t)$ in $\on{Diff}_{c}(\R^d)$ starting at $\on{Id}_{\R^d}$ and satisfying $c_2 = \ph(1) \circ c_1$ with length $\on{Len}^{\mc D}(\ph) < R$. Define the new path $\et(t) = \ph(t) \circ \ps$. Then $c_1 = \et(0) \circ c_0$ and $c_2 = \et(1) \circ c_0$ and because of the estimate
\begin{align*}
\on{dist}^{\mc D}(\on{Id}_{\R^d}, \et(t)) 
&\leq \on{dist}^{\mc D}(\on{Id}_{\R^d}, \ps)  + \on{dist}^{\mc D}(\ps, \ph(t) \circ \ps) \\
&= \on{dist}^{\mc D}(\on{Id}_{\R^d}, \ps)  + \on{dist}^{\mc D}(\on{Id}_{\R^d}, \ph(t)) \\
&< 2R\,,
\end{align*}
the path remains in $B^{\mc D}(\on{Id}_{\R^d},2R)$. Thus by Lemma~\ref{lem:diff_ball_bound} the quantity $\|\et(t) - \on{Id}\|_{H^s(\R^d)}$ is bounded from above and $\inf_{x \in \R^d} \det D\et(t,x)$ from below along the path. 

Define the path $c(t) = \et(t) \circ c_0$, which connects $c_1$ and $c_2$. Our aim is to estimate $\on{Len}^{\mc I}(c)$ in terms of $\on{Len}^{\mc D}(\ph)$. Let $s' = s-(d-1)/2$ be the Sobolev order of the induced metric $G^{\mc O}$. By assumption $n \leq s'$. From the inequality
\begin{align*}
\| c(t) \|_{H^n(d\th)} 
&= \left\| (\et(t) - \on{Id}) \circ c_0 + c_0 \right\|_{H^n(d\th)} \\
&\leq C_1 \left\| \et(t) - \on{Id} \right\|_{H^s(\R^d)} + \| c_0 \|_{H^n(d\th)}\,,
&&\text{ via Lemma~\ref{lem:trace_theorem}}
\end{align*}
we see that $\| c(t) \|_{H^n(d\th)}$ is also bounded along the path. Next we obtain from $c'(t,\th) = D\et(t,c_0(\th)).c_0'(\th)$ the bound
\[
|c'(t,\th)| \geq \frac{1}{\| \left(D\et(t)\right)\inv\|_\infty} |c_0'(\th)|\,,
\]
and thus $\inf_{\th} |c'(t,\th)|$ is bounded from below along the path. This allows us to apply Lemma~\ref{lem:curve_norm_equiv} along the path.

Now we can estimate the $G^{\mc I}$-geodesic distance between $c_1$ and $c_2$,
\begin{align*}
\on{dist}_{\mc I}(c_1, c_2)
&\leq \on{Len}_{\mc I}(c) = \int_0^1 \sqrt{G_{c,\mc I}(\p_t c, \p_t c)} \ud t \\
&\leq C_1 \int_0^1 \| \p_t c(t) \|_{H^n(d\th)} \ud t 
&&\text{via Lemma~\ref{lem:curve_norm_equiv}} \\
&= C_1 \int_0^1 \| \p_t \et(t) \circ c_0 \|_{H^n(d\th)} \ud t \\
& \leq C_2 \int_0^1 \| \p_t \et(t) \|_{H^s(\R^d)} \ud t
&&\text{via Lemma~\ref{lem:trace_theorem}} \\
&\leq C_3 \int_0^1 \| \p_t \et(t) \circ \et(t)\inv\|_{H^s(\R^d)} \ud t
&&\text{via Lemma~\ref{lem:comp_diff_ball}} \\
&= C_3 \on{Len}^{\mc D}(\et) = C_3 \on{Len}^{\mc D}(\ph)\,,
\end{align*}
and the constant $C_3$ depends, among other things, on $c_0$ and $R$. By taking the infimum over all paths $\ph(t)$, we obtain the desired inequality
\[
\on{dist}_{\mc I}(c_1, c_2) \leq C_3 \on{dist}^{\mc D}(c_1, c_2)\,.\qedhere
\]
\end{proof}

Shadowing the structure of Propositions~\ref{prop:flat_outer_bound} and~\ref{prop:outer_flat_bound} we can also prove the other inequality, which holds locally around each curve.

\begin{proposition}
\label{prop:outer_inner_bound}
Let $G^{\mc I}$ and $G^{\mc O}$ be as in Theorem~\ref{thm:inner_outer_bound}. Given $c_0 \in \on{Emb}(S^1,\R^d)$, there exists $R=R(c_0) > 0$ and $C=C(c_0) > 0$ such that
\[
\on{dist}^{\mc O}(c_1,c_2) \leq C \on{dist}^{\mc I}(c_1,c_2)\,,
\]
for all $c_1,c_2 \in B^{\mc I}(c_0,R)$.
\end{proposition}

\begin{proof}
The proof is a combination of Proposition~\ref{prop:outer_flat_bound} and~\cite[Lemma~4.2(1)]{Bruveris2015}. Given $c_0$, using Proposition~\ref{prop:outer_flat_bound} we obtain $R_1=R_1(c_0)$ such that
\[
\on{dist}^{\mc O}(c_1,c_2) \leq C_1 \| c_1 - c_2 \|_{H^{n}(S^1)}\,,
\]
for some $C_1=C_1(c_0)$ and $c_1,c_2$ such that $\| c_i - c_0 \|_{H^n} \leq R_1$. Next we apply~\cite[Lemma~4.2(1)]{Bruveris2015} to obtain the inequality
\[
\|c_1-c_2 \|_{H^n(S^1)} \leq C_2 \on{dist}^{\mc I}(c_1,c_2)\,,
\]
for some $C_2=C_2(c_0,R_2)$ and $c_1,c_2 \in B^{\mc I}(c_0,R_2)$ with $R_2$ chosen arbitrary. Hence, if we choose $R = \min(C_2\inv R_1, R_2)$, then $c_1 \in B^{\mc I}(c_0,R)$ implies $\| c_1 - c_0\|_{H^{n}(S^1)} \leq R_1$ and hence
\[
\on{dist}^{\mc O}(c_1,c_2) \leq C_1 C_2 \on{dist}^{\mc I}(c_1,c_2)\,,
\]
for all $c_1,c_2 \in B^{\mc I}(c_0,R)$.
\end{proof}

\let\i=\ii           
\printbibliography

\end{document}